\documentclass{amsart}
\usepackage{amsmath}
\usepackage{amssymb}
\usepackage{amsthm}
\usepackage{mathrsfs}

\newtheorem{theorem}{Theorem}[section]

\theoremstyle{definition}  

\newtheorem{Question}[theorem]{Question}

\theoremstyle{remark}

\numberwithin{equation}{section}

\begin{document}

\title{Intrinsic characterizations of $C$-realcompact spaces}

\author{ Sudip Kumar Acharyya}
\address{Department of Pure Mathematics, University of Calcutta, 35, Ballygunge Circular Road, Kolkata - 700019, INDIA} 
\email{sdpacharyya@gmail.com}

\author{Rakesh Bharati}
\address{Department of Pure Mathematics, University of Calcutta, 35, Ballygunge Circular Road, Kolkata - 700019, INDIA} 
\email{bharti.rakesh292@gmail.com}

\thanks{The second author acknowledges financial support from University Grand Commission, New Delhi, for the award of research fellowship (F. No. 16-9(June 2018)/2019 (NET/CSIR))}

\author{A. Deb Ray }
\address{Department of Pure Mathematics, University of Calcutta, 35, Ballygunge Circular Road, Kolkata - 700019, INDIA} 
\email{debrayatasi@gmail.com}

\begin{abstract}
$c$-realcompact spaces are introduced by Karamzadeh and Keshtkar in Quaest. Math. 41(8), 2018, 1135-1167. We offer a characterization of these spaces $X$ via $c$-stable family of closed sets in $X$ by showing that  $X$ is $c$-realcompact if and only if each $c$-stable family of closed sets in $X$ with finite intersection property has nonempty intersection. This last condition which makes sense for an arbitrary topological space can be taken as an alternative definition of a $c$-realcompact space. We show that each topological space can be extended as a dense subspace to a $c$-realcompact space with some desired extension property. An allied class of spaces viz $CP$-compact spaces akin to that of $c$-realcompact spaces are introduced. The paper ends after examining how far a known class of $c$-realcompact spaces could be realized as $CP$-compact for appropriately chosen ideal $P$ of closed sets in $X$.
\end{abstract}

\maketitle

\section{Introduction}
\noindent In what follows $X$ stands for a completely regular Hausdorff topological space. As usual $C(X)$ and $C^*(X)$ denote respectively the ring of all real valued continuous functions on $X$ and that of all bounded real valued continuous functions on $X$. Suppose $C_c(X)$ is the subring of $C(X)$ containing those functions $f$ for which $f(X)$ is a countable set and $C_c^*(X)= C_c(X)\cap C^*(X)$. Formal investigations of these two rings vis-a-vis the topological structure of $X$ are being carried on only in the recent times. It turns out that there is an interplay between the topological structure of $X$ and the ring and lattice structure of $C_c(X)$ and $C_c^*(X)$, which incidentally sheds much light on the topology of $X$. The articles \cite{ref3}, \cite{ref4}, \cite{ref7}, \cite{ref8}, \cite{ref11} may be referred in this context. The notion of $c$-realcompact spaces is the fruit of one such endeavours in the study of $X$ versus $C_c(X)$ or $C_c^*(X)$. A space $X$ is declared $c$-realcompact in \cite{ref8} if each real maximal ideal $M$ in $C_c(X)$ is fixed in the sense that there exits a point $x\in X$ such that for each $f \in M$, $f(x)=0$. $M$ is called real when the residue class field $C_c(X)/M$ is isomorphic to the field $\mathbb{R}$. A number of interesting facts concerning these spaces is discovered in \cite{ref8}. These may be called countable analogues of the corresponding properties of real compact spaces as developed in \cite{ref6}, chapter 8. In the present article we offer a new characterization of $c$-realcompact spaces on using the notion, $c$-stable family of closed sets in $X$. A family $\mathcal{F}$ of subsets of $X$ is called $c$-stable if given $f\in C(X,\mathbb{Z})$, there exists $F\in \mathcal{F}$ such that $f$ is bounded on $F$.

We define a topological space $X$ (not necessary completely regular ) to be $c_c$-realcompact if each $c$-stable family of closed sets in $X$ with finite intersection property has nonempty intersection. We check that this new notion of $c_c$-realcompactness agrees with the already introduced notion of $c$-realcompactness in \cite{ref8}, within the class of zero-dimensional Hausdorff spaces (Theorem \ref{2.3}). We re-establish a modified version of a few known properties of $c$-realcompact spaces using our new definition $c_c$-raelcompactness (Theorem \ref{2.4}). Furthermore we realize that any topological space $X$ can be extended as a dense subspace to a $c_c$-realcompact space $\upsilon_0X$ enjoying some desired extension properties (Theorem \ref{2.5}). While constructing this extension of $X$, we follow closely the technique adopted in \cite{ref9}. The results mentioned above constitute the first technical section viz $\S2$ of this article.

A family $\mathcal{P}$ of closed sets in $X$ is called an ideal of closed sets if $A\in \mathcal{ P}$, $B\in \mathcal{ P}$ and $C$ is a closed subset of $A$ imply that $A\cup B\in \mathcal{ P}$ and $C\in \mathcal{ P}$. Let $\Omega(X)$ stand for the aggregate of all ideals of closed sets in $X$. For any $\mathcal{ P}\in \Omega(X)$ let $C_\mathcal{ P}(X)=\{f \in C(X):cl_X(X \setminus Z(f))\in \mathcal{P}\}$, here $Z(f)=\{x\in X:f(x)=0\}$ is the zero set of $f$ in $X$. It is well known that $C_\mathcal{ P}(X)$ is an ideal in the ring $C(X)$, see \cite{ref1} and \cite{ref2} for more information on these ideals. With referance to any such $\mathcal{ P}\in \Omega(X) $, we call a family $\mathcal{F}$ of subsets of $X$ $c_\mathcal{ P}$-stable if given $f\in C(X,\mathbb{Z})\cap C_\mathcal{ P}(X)$ there exists $F\in \mathcal{F}$ such that $f$ is bounded on $F$. We define a space $X$ to be $c_\mathcal{ P}$-compact if any $c_\mathcal{ P}$-stable family of closed sets in $X$ with finite intersection property has non-empty intersection. It is clear that a zero-dimensional space $X$ is $c_c$-realcompact if it is already $c_\mathcal{ P}$-compact.

We have shown that if $X$ is a noncompact zero-dimensional space and $\mathcal{ P}\in \Omega(X)$ such that $X$ is $c_\mathcal{ P}$-compact, then there exists an $\mathcal{R}\in \Omega(X)$ such that $\mathcal{R} \varsubsetneqq \mathcal{ P}$ and $X$ is $c_\mathcal{R}$-compact. Thus within the class of zero-dimensional noncompact spaces $X$, there is no minimal member $\mathcal{ P}\in \Omega(X)$ in the set inclusion sense of the term for which $X$ becomes $c_\mathcal{ P}$-compact (Theorem \ref{t-3.3}). In the concluding portion of $\S3$ of this article we have examined, how far the known classes of $c$-realcompact spaces could be achieved as $c_\mathcal{ P}$-compact spaces for appropriately chosen $\mathcal{ P}\in \Omega(X)$. For any infinite cardinal number $\theta$, $X$ is called finally $\theta$-compact if each open cover of $X$ has a subcover with cardinality $<\theta $ (see \cite{ref10}). In this terminology finally $\omega_1$-compact spaces are Lindel\"of and finally $\omega_0$-compact spaces are compact. It is realized that a $c$-realcompact space $X$ is finally $\theta$-compact if and only if it is $c_\mathcal{Q}$-compact, where $\mathcal{Q}$ is the ideal of all closed finally $\theta$-compact subsets of $X$ (Theorem \ref{t-3.5}). A special case of this result reads: $X$ is Lindel{\"o}f when and only when $X$ is $c_\alpha $-compact where $\alpha$ is the ideal of all closed Lindel{\"o}f subsets of $X$.

\section{Properties of $c_c$-realcompact spaces and $c_c$-realcompactifications}

Before stating the first technical result of this section, we need to recall a few terminologies and results from \cite{ref4} and \cite{ref8}. Our intention is to make the present article self contained as far as possible. An element $\alpha$ on a totally ordered field $F$ is called infinitely large if $\alpha>n$ for each $n\in \mathbb{N}$. It is clear that $F$ is archimedean if and only if it does not contain any infinitely large element. If $M$ is maximal ideal in $C_c(X)$ then the residue class field $C_c(X)/M$ is totally ordered according to the following definition: for $f\in C(X)$ $M(f)\geqq 0$ if and only if the exists $g\in M$ such that $f\geqq 0$ on $Z(g)$. Here $M(f)$ stands for the residue class in $C(X)/M$, which contains the function $f$.

\begin{theorem}\label{t-2.1}
	\textnormal{(}Proposition 2.3 in \cite{ref8}\textnormal{)} For a maximal ideal $M$ in $C_c(X)$ and for $f\in C_c(X)$, $|M(f)|$ is infinitely large in $C_c(X)/M$ if and only if $f$ is unbounded on every zero set of $Z_c(M)=\{Z(g):g\in M\}$.
\end{theorem} It is proved in \cite{ref4}, Remark 3.6 that if $X$ is a zero-dimensional space, then the set of all maximal ideals of $C_c(X)$ equipped with hull-kernel topology, also called the structure space of $C_c(X)$ is homeomorphic to the Banaschewski compactification $\beta_0X$ of $X$. Thus the maximal ideals of $C_c(X)$ can be indexed by virtue of the points of $\beta_0X$. Indeed a complete description of all these maximal ideals  is given by the list $\{M_c^p:p\in \beta_0X\}$, where $M_c^p=\{f\in C_c(X):p\in cl_{\beta_0X}Z(f)\}$ with $M_c^p$ is a fixed maximal ideal if and only if $ p\in X$ (see Theorem 4.2 in \cite{ref4}). It is well known that any continuous map $f:X\rightarrow Y$, where $X$ and $Y$ are both zero-dimensional spaces with $Y$ compact also, has an extension to a continuous map $\bar{f}:\beta_0X\rightarrow Y$ (we call this property, the $C$-extension property of $\beta_0X$) (see Remark 3.6 in \cite{ref4}). It follows that for a zero-dimensional space $X$, any continuous map $f:X\rightarrow \mathbb{Z}$ (also written as $f\in C(X,\mathbb{Z}))$, has an extension to a continuous map $f^*:\beta_0X\rightarrow\mathbb{Z}^*=\mathbb{Z}\cup \{\omega\}$, the one point compactification of $\mathbb{Z}$. We also write $f^*\in C(\beta_0X,\mathbb{Z}^*)$. A slightly variant form of the next result is proved in \cite{ref8}, Theorem 2.17 and Theorem 2.18.
\begin{theorem}\label{2.2}
	Let $X$ be zero-dimensional and $p\in \beta_0X$, then the maximal ideal $M_c^p$ in $C_c(X)$ is real if and only if for each $f\in C(X,\mathbb{Z})$, $f^*(p)\neq\omega$ if and only if $|M_c^P(f)|$ is not infinitely large in $C_c(X)/M_c^p$.
\end{theorem}

\begin{theorem}\label{2.3}
	A zero-dimensional space $X$ is $c_c$-realcompact if and only if it is $c$-realcompact.
\end{theorem}
\begin{proof}
	Let $X$ be a $c$-realcompact space and $\mathcal{F}$ be a family of closed subsets of $X$ with finite intersection property but with $\bigcap$ $\mathcal{F}$=$\emptyset$. To show that $X$ is $c_c$-realcompact we shall prove that $\mathcal{F}$ is not a $c$-stable family. Indeed $\{cl_{\beta_0X} F:F\in \mathcal{F}\}$ is a family of closed subsets of $\beta_0X$ with finite intersection property. Since $\beta_0X$ is compact, there exists a point $p\in \bigcap\limits_{F\in \mathcal{F}} {cl_{\beta_0X} F}$ and of course $p\in \beta_0X\setminus X$. Here $M_c^p$ is a free maximal ideal in $C_c(X)$. Since $X$ is $c$-realcompact this implies that $M_c^p$ is a hyperreal maximal ideal (meaning that it is not a real maximal ideal of $C_c(X)$). It follows from Theorem \ref{2.2} that there exists $f\in C(X,\mathbb{Z})$ with $f^*(p)=\omega$. Since $p\in cl_{\beta_0X} F$ for each $F\in \mathcal{F}$, it is therefore clear that `$f$' is unbounded on each set in the family $\mathcal{F}$. Therefore $\mathcal{F}$ is not a $c$-stable family.
	
	Conversely let $X$ be not $c$-realcompact. Then there exists a real maximal ideal $M$ in $C_c(X)$, which is not fixed. This means that there is a point $p\in \beta_0X\setminus X$ for which $M=M_c^p$. Since $p\in cl_{\beta_0X}Z(f)$ for each $f\in M_c^p$, it follows that $\{Z(f):f\in M_c^p\}$ is a family of closed sets in $X$ with finite intersection property but with empty intersection. To show that $X$ is not $c_c$-realcompact, it suffices to show that $\{Z(f):f\in M_c^p\}$ is a $c$-stable family. So let $g\in C(X,\mathbb{Z})$. Since $M_c^p$ is real, this implies in view of Theorem \ref{2.2} that $g^*(p)\neq\omega$ and hence $|M_c^p(g)|$ is not infinitely large. It follows therefore from Theorem \ref{t-2.1} that $g$ is bounded on some $Z(f)$ for an $f\in M_c^p$. This settles that $\{Z(f):f\in M_c^p\}$ is a $c$-stable family.
\end{proof}
	
	By adapting the arguments of Theorem 5.2, Theorem 5.3 and Theorem 5.4 in \cite{ref9}, appropriately we can establish the following facts about $c_c$-realcompact spaces without difficulty:
\begin{theorem} \label{2.4}

	\begin{enumerate}

	\item A compact space is $c_c$-realcompact.
	\item A pseudocompact $c_c$-realcompact space is compact.
	\item A closed subspace of a $c_c$-realcompact space is $c_c$-realcompact.
	\item The product of any set of $c_c$-realcompact spaces is $c_c$-realcompact.
	\item If a topological space $X=E\cup F$ where $E$ is a compact subset of X and $F$ is a $\mathbb{Z}$-embedded $c_c$-realcompact subset of $X$, meaning that each function in $C(F, \mathbb{Z})$ can be extended to a function in  $C(X,\mathbb{Z})$, then $X$ is $c_c$-realcompact.
	\item A $\mathbb{Z}$-embedded $c_c$-realcompact subset of a Hausdorff space $X$ is a closed subset of $X$.
	\end{enumerate}
\end{theorem}	
We now show that any topological $X$ can be extended to a $c_c$-realcompact space containing the original space $X$ as a $C$-embedded dense subspace and enjoying a desirable extension property. The proof can be accomplished by closely following the arguments adopted to prove  Theorem 6.1 in \cite{ref9}. Nevertheless we give a brief outline of the main points of proof in our theorem.
\begin{theorem}\label{2.5}
	Every topological space $X$ can be extended to a $c_c$-realcompact space $\upsilon^c X$ as a dense subspace with the following extension property: each continuous map from $X$ into a regular $c_c$-realcompact space $Y$ can be extended to a continuous map from $\upsilon^cX$ into $Y$. $X$ is $c_c$-realcompact if and only if $X=\upsilon^cX$
\end{theorem} 
\begin{proof}
	For each $x\in X$ let $\mathcal{ G}^x$ be the aggregate of all closed sets in $X$ which contain the point $x$. Then $\mathcal{ G}^x$ is a $c$-stable family of closed sets in $X$ with finite intersection property and with the prime condition: $A\cup B\in \mathcal{ G}^x\implies A\in \mathcal{ G}^x $ or $ B\in \mathcal{ G}^x$, $A,B\subseteq X$. We extend the set $X$ to a bigger set $\upsilon^c X$, so that $\upsilon^c X\setminus X$ becomes an index set for the collection of all maximal $c$-stable families of closed subsets of $X$ with finite intersection property but with empty intersection. For each $p\in \upsilon^c X\setminus X$, let $\mathcal{ G}^p$ designate the corresponding maximal $c$-stable family of closed sets in $X$ with finite intersection property and with empty intersection. For each closed set $F$ in $X$, we write $\bar{F}=\{p\in \upsilon^cX:F\in \mathcal{ G}^p\}$. Then $\{\bar{F}$: $F$is closed in $X$\} forms a base for closed sets of some topology on  $\upsilon^c X $ and in this topology for any closed set $F$ in $X$ $\bar{F}=cl_{\upsilon^cX} F$. Since $X$ belongs to each $\mathcal{ G}^p$, it is clear that $X$ is dense in $\upsilon^c X $. Let $t:X\rightarrow Y$ be a continuous map with $Y$, a regular $c_c$-realcompact space. Choose $p\in \upsilon^c X $. Let $\mathcal{H}^p=\{G\subseteq Y: G$ is closed in Y and $t^{-1}(G)\in\mathcal{G}^p\}$. Then $\mathcal{H}^p$ is a $c$-stable family of closed sets in $Y$ with finite intersection property. We select a point $y\in \bigcap\mathcal{H}^p$ and we set $t^0(p)=y$ with the aggrement that $t^0(p)=t(p)$ in case $p\in X$. Thus $t^0:\upsilon^c X\rightarrow Y$ is a well defined map which is further continuous. The remaining parts of the theorem can be proved by making arguments closely as in the proof of Theorem 6.1 of \cite{ref9}.
\end{proof}

	\section{$c_\mathcal{ P}$-compact spaces}
	In this section all the topological spaces $X$ that will appear will be assumed to be zero-dimensional. We define for any $\mathcal{ P}\in \Omega(X)$, $\upsilon_0^{\mathcal{P}}(X)=\{p\in \beta_0X:f^*(p)\neq\omega$ for each $f\in C_\mathcal{ P}(X)\cap C(X,\mathbb{Z})\}$. It is clear that if $\mathcal{ P}$=$\mathcal{E}$$\equiv$ the ideal of all closed sets in $X$ then $\upsilon_0^{\mathcal{E}}(X)=\upsilon_0X\equiv\{p\in \beta_0X:f^*(p)\neq\omega$ for each $f\in C(X,\mathbb{Z})\}$ the set defined in the begining of the proof of Theorem 3.8 in \cite{ref8}. The next theorem puts Theorem \ref{2.3} in a more general setting.
	\begin{theorem}\label{t-3.1}
For a $\mathcal{P}\in \Omega$, $X$ is $c_\mathcal{P}$-compact if and only if $X=\upsilon_0^{\mathcal{ P}}(X)$	
\end{theorem}
	We omit the proof of this theorem because it can be done by making some appropriate modification in the arguments adopted in the proof of Theorem \ref{2.3}.
	\\
	It is clear that if $\mathcal{ P}$, $\mathcal{Q}\in \Omega(X)$ with $\mathcal{ P}\subset\mathcal{Q}$, then any $c_\mathcal{Q}$-stable family of closed sets in $X$ is also $c_\mathcal{P}$-stable, consequently if $X$ is $c_\mathcal{P}$-compact then $X$ is $c_\mathcal{Q}$-compact also. In particular every $c_\mathcal{ P}$-compact space is $c_c$-realcompact and hence $c$-realcompact in view of Theorem \ref{2.3}. The following question therefore seems to be natural.
	\begin{Question}
		If $X$ is a zero-dimensional non-compat $c$-realcompact space, then does there exist a minimal ideal $\mathcal{ P}$ of closed sets in $X$ (minimal in some sense of the term) for which $X$ becomes $c_\mathcal{ P}$-compact ?
	\end{Question}
		
		No possible answer to this question is known to us, however the following proposition shows that the answer to this question is in the negative if the phrase `minimal' is interpreted in the set inclusion sense of the term.
		
		\begin{theorem}\label{t-3.3}
			Let $X$ be a non compact zero-dimensional space. Suppose $\mathcal{P} \in \Omega(X)$ is such that $X$ is $c_\mathcal{P}$-compact. Then there exists $\mathcal{R}\in \Omega(X)$ such that $\mathcal{R}  \varsubsetneqq\mathcal{P}$ and $X$ is $c_\mathcal{R}$-compact.
		\end{theorem}
\begin{proof}
We get from Theorem \ref{t-3.1} that $X=\upsilon_0^\mathcal{P}X$. As $X$ is non compact we can choose a point $p\in \beta_0X\setminus X$. Then $p\notin \upsilon_0^\mathcal{P}X$. Accordingly there exists $f\in C_\mathcal{P}(X)\cap C(X,\mathbb{Z})$ such that $f^*(p)=\omega$. We select a point $x\in X$ such that $f(x)\neq 0$. Set $\mathcal{R}=\{D\in \mathcal{P}:x\notin D\}$. It is easy to check that $\mathcal{R}$ is an ideal of closed sets in $X$ i.e., $\mathcal{R}\in \Omega (X)$. Furthermore, $cl_X(X-Z(f))$ is a member of $\mathcal{P}$ containing the point $x$. This implies that $cl_X(X-Z(f))\notin \mathcal{R}$. Thus $\mathcal{R}\varsubsetneq\mathcal{P}$. To show that $X$ is $c_\mathcal{R}$-compact. We shall show that $X=\upsilon_0^\mathcal{R}X $ (see Theorem \ref{t-3.1}). So choose a point $q\in \beta_0X\setminus X$ then $q\notin \upsilon_0^\mathcal{P}X$, consequently there exists $g\in C_\mathcal{P}(X)\cap C(X,\mathbb{Z})$ such that $g^*(q)=\omega$. For the distinct points $q, x$ in $\beta_0X$ there exist disjoint open sets $U$, $V$ in this space such that $x\in U$, $q\in V$. Since $\beta_0X$ is zero-dimensional there exists therefore a clopen set $W$ in $\beta_0X$ such that $q\in W\subset V$. The map $h:\beta_0X\rightarrow \{0,1\}$ given by $h(W)=\{1\}$ and $h(\beta_0X\setminus W)=\{0\}$ is continuous. We note that $h(U)=\{0\}$ and $h(q)=1$. Let $\psi=h|_X$. Then $\psi \in C(X,\mathbb{Z})$. Take $l=g.\psi$. Since $g\in C_\mathcal{P}(X)$ and $C_\mathcal{P}(X)$ is an ideal of $C(X)$, it follows that $l\in C_\mathcal{P}(X)$. Furthermore the fact that $g$ and $\psi$ are both functions in $C(X,\mathbb{Z})$ implies that $l\in (X,\mathbb{Z})$. Also the function $h\in C(\beta_0X, \mathbb{Z})$ is the unique continuous extension of $\psi \in C(X, \mathbb{Z})$, hence we can write $h=\psi^*$. This implies that $l^*(q)=g^*(q)\psi^*(q)=g^*(q)h(q)=\omega$, because $g^*(q)=\omega$ and $h(q)\neq0$
\\
On the other hand if $y\in U\cap X$ then $h(y)=0$ and hence $l(y)=0$. Since $U\cap X$ is an open neighbourhood of $x$ in the space $X$, this implies that $x\notin cl_X(X\setminus Z(l))$. Since $cl_X(X\setminus Z(l))\in \mathcal{ P}$, already verified, it follows that $cl_X(X\setminus Z(l))\in \mathcal{R}$. Thus $l\in C_\mathcal{R}(X)\cap C(X,\mathbb{Z})$. Since $l^*(q)=\omega$, this further implies that $q\notin \upsilon_0^\mathcal{R}(X)$.
\end{proof}

It is trivial that a (zero-dimensional) compact space is $c$-realcompact. It is also observed that a Lindel{\"o}f space is $c$-realcompact (Corollary 3.6, \cite{ref8}). But for an infinite cardinal number $\theta$, a finally $\theta$-compact space may not be $c$-realcompact. Indeed the space $[0,\omega_1)$ of all countable ordinals is a celebrated example of a zero-dimensional space which is not realcompact (see 8.1, \cite{ref6}). Since a zero-dimensional $c$-realcompact space is necessarily realcompact (vide proposition 5.8, \cite{ref8}) it follows therefore that [$0,\omega_1$) is not a $c$-realcompact space. But it is easy to show that [$0,\omega_1$) is finally $\omega_2$-compact. For the same reason, the Tychonoff plank $T\equiv$[$0,\omega_1$)$\times$[$0,\omega_0$)-$\{(\omega_1,\omega_0)\}$ of 8.20 in \cite{ref6}, is finally $\omega_2$-compact without being $c$-realcompact. It can be easily shown that a closed subset of a finally $\theta$-compact space is finally $\theta$-compact.
\\

Furthermore, the following characterization of finally $\theta$-compactness of a topological space can be established by routine arguments.
\begin{theorem}\label{t-3.4}
The following two statements are equivalent for an infinite cardinal number $\theta$.
\begin{enumerate}
\item $X$ is finally $\theta$-compact.
\item If $\mathcal{B}$ is a family of closed sets in $X$, such that for any subfamily $\mathcal{B}_0$ of $\mathcal{B}$ with $|\mathcal{B}_0|<\theta$, $\bigcap \mathcal{B}_0\neq\emptyset$, then $\bigcap \mathcal{B}\neq\emptyset$.
\end{enumerate}
\end{theorem}

\begin{theorem}\label{t-3.5}
Let $X$ be $c$-realcompact and $\mathcal{P}_\mathcal{\theta}$ the ideal of all closed finally $\theta$-compact subsets of $X$. Then $X$ is finally $\theta$-compact if and only if it is $c_{\mathcal{P}_\theta}$-compact.
\end{theorem}

\begin{proof}
Let $X$ be finally $\theta$-compact and $p\in \beta_0X \setminus X$. To show that $X$ is $c_{\mathcal{P}_\theta}$-compact it suffices to show in view of Theorem \ref{t-3.1} that $p\notin \upsilon_0^{\mathcal{P}_\theta}(X)$. Indeed $X$ is $c$-realcompact implies that the maximal ideal $M_c^p$ of $C_c(X)$ is not real. Consequently by Theorem \ref{2.2}, there exists $f\in C(X,\mathbb{Z})$ such that $f^*(p)=\omega$. Now $cl_X(X\setminus Z(f))$, like any closed subsets of $X$ is finally $\theta$-compact. Thus $f\in C_{\mathcal{P}_\theta} (X)\cap C(X,\mathbb{Z})$, hence $p\notin \upsilon _0^{\mathcal{P}_\theta}(X)$.
\\
To prove the converse, let $X$ be not finally $\theta$-compact. It follows from Theorem \ref{t-3.4} that there exists a family $\mathcal{B}=\{B_\alpha : \alpha \in \Lambda \}$ of closed sets in $X$ with the following properties: for any subfamily $\mathcal{B}_1$ of $\mathcal{B}$ with $|\mathcal{B}_1|< \theta$. $\bigcap \mathcal{B}_1\neq \emptyset$ but $\bigcap \mathcal{B}=\emptyset$. Let $\mathcal{D}=\{D_\alpha:\alpha\in \Lambda^* \}$ be the aggregate of all sets $D_\alpha's$, which are intersections of $<\theta$ many sets in the family $\mathcal{B}$. Then $\mathcal{B}\subseteq \mathcal{D}$ and hence $\bigcap \mathcal{D}=\emptyset$. Also $\mathcal{D}$ has finite intersection property. We shall show that $\mathcal{D}$ is a $c_{\mathcal{P}_\theta}$-stable family and hence $X$ is not $\mathcal{P}_\theta$-compact. Towards such a proof choose $f\in C_{\mathcal{P}_\theta}(X)\cap C(X,\mathbb{Z})$, then $cl_X(X\setminus Z(f)$ is a finally $\theta$-compact subset of $X$. Since $\{X\setminus B_\alpha:\alpha \in \Lambda\}$ is an open cover of $X$, there exists a subset $\Lambda_0$ of $\Lambda$ with $|\Lambda_0|<\theta$ such that $cl_X(X\setminus Z(f))\subseteq \bigcup\limits_{\alpha\in \Lambda_0} (X\setminus B_\alpha)$. This implies that $\bigcap\limits_{\alpha\in \Lambda_0} B_\alpha \subseteq Z(f)$ and we note that $\bigcap\limits_{\alpha\in \Lambda_0} B_\alpha \in \mathcal{D}$. Thus $f$ becomes bounded on a set lying in the family $\mathcal{D}$. Hence $\mathcal{D}$ becomes a $c_{\mathcal{P}_\theta}$-stable family.
\end{proof}

\end{document}